\documentclass[12pt]{article}
\usepackage{amsmath, amsfonts, amssymb, amsthm, epsfig, epstopdf, titling, url, array}
\usepackage[utf8]{inputenc}
\usepackage{mathtools}
\usepackage[english]{babel}
\usepackage{hyperref}
\usepackage{float}
\newtheorem{thm}{Theorem}[section]
\newtheorem{lem}{Lemma}[section]

\numberwithin{equation}{section}

\def\C{\mathbb C}

\def\R{I\!\!R}

\def\C{I\!\!\!\!C}

\title{Integral Formulas and asymptotic behavior of lattice points in complex hyperbolic space}
\author{Mohamed Vall Ould Moustapha}

\begin{document}
\maketitle
\begin{abstract}
This paper deals with the $\Gamma$-lattice points problem associated to a discrete subgroup of motions $\Gamma$  in the complex hyperbolic space $\C H^n$.  We give two integral formulas for the local average of the  number $N(T, z, z')$ of
 $\Gamma$- lattice points in a sphere of radius $T$ in $\C H^n$. The first on is  in terms of the
 solution of the $\Gamma$-automorphic wave equation on $\C H^n$ and the second is given in terms of the spectral function of the Laplace-Beltrami operator under 
$\Gamma$-automorphic boundary conditions. We use the obtained integral formulas to obtain an asymptotic behavior of the number $N(T, z, z')$ as
 $T\rightarrow \infty$, with an estimate of the remainder term. Our principal tools are the explicit solution of the wave equation on the complex hyperbolic space,  
special functions and spectral theory of the Laplace-Beltrami operator under $\Gamma$  automorphic boundary conditions.
\end{abstract}
\section{Introduction}
The lattice points are the orbit of the origin under the action of the group
of integer translations on the Euclidean plane. The classical Gauss Circle Problem is to determine the best bound for the error
between the number of lattice points inside a disk and that disk’s area, otherwise
known as the lattice point discrepancy.
 Gauss had proved  that the number of lattice points in a circle of large radius $T$, was equal to the area of the disc 
with a remainder term not exceeding the circumference of the circle  (see Gauss \cite{Gauss} and Hafner \cite{Hafner}).
 Since then lattice points counting problems in Euclidean plane and more generally in Eulidean space, have been considered by many authors, with various applications in number theory
(\cite{Colin de Verdiere, Kai, Kratzel1, Kratzel2, Landau1, Landau2, Narkiewicz-Wladyslaw, Randol1, Randol2, Walfisz}).
 The Euclidean  space can be replaced by 
any Riemannian space $X$ and the group of integer translations can be replaced by
any discrete subgroup $\Gamma$ of the group
of motions $G$  of the Riemannian space $X$.\\
The case of lattice points problem in the hyperbolic setting was first studied by Jean Delsarte who obtained in (Delsarte \cite{Delsarte1} and \cite{Delsarte2})
the following remarkable formula for the number of lattice points $N(T, z, z')$ in the hyperbolic circle of radius $T$.
\begin{equation} N(T, z, z')=\pi z\label{Delsarte-Formula}\sum_{n=0}^{+\infty}F((\alpha_n, \beta_n, 2, -z/4)\varphi_n(w_0)\varphi_n(w), \end{equation}  where 
$z=2a^2(\cosh T/a-1)$ and $\alpha_n, \beta_n$ are solution of the equation\\ $x^2-x-\lambda_n a^2=0,$
 and $\varphi_n(w)$  are the eigenfunctions of the Laplace-Beltrami on a compact fundamental region of the hyperbolic
plane and $\lambda_n$ are the corresponding
eigenvalues,  and $F(a, b, c, t)$ is the classical Gauss hypergeometric function ${}_2F_1$,
 defined by (Magnus et al. \cite{Magnus et al.} p. 37)
\begin{align}\label{Gauss}
  F(a, b, c, z)=\sum_{n=0}^{\infty}\frac{(a)_n(b)_n}{(c)_n n!}z^n,
  \quad |z|<1,
\end{align}
 $(a)_n$ is the Pochhamer symbol
  $(a)_n=\frac{\Gamma(a+n)}{\Gamma(a}$
and $\Gamma$ is the classical Euler function.\\
Many other authors have studied the asymptotic behavior of the number $N(T, z, z')$ as $T 
\longrightarrow \infty$, in the cases of non Euclidean spaces 
 (\cite{Alsina-Chatzakos, Arkhipov-Chubatikov, Berard, Biro,  Bruce et al., Elstrodt et al., Hill-Parnovski,  Huber1, Huber2, Kelly, Parkkonen, Patterson,  Petridis-Risage, Phillips-Rudnick,  Selberg, Soundararajan, Wolfe}).\\
For lattice points in non compact type symmetric spaces of rank one (see\cite{B-G-M, B-M-W, Gunther, Miatello-Wallach}). 
Lax and Phillips \cite{Lax-Phillips} gave a formula expressing the 
average of number of lattice points, for the real hyperbolic case, in terms of the solution of real 
hyperbolic wave equation for a special initial function.
Levitan \cite{Levitan} derived a formula which gives an expression for the average of number of lattice 
points in real hyperbolic space, in terms of the spectral function of the shifted 
Laplace-Beltrami operator on a fundamental region $F=\R H^n/\Gamma$. He 
estimated the remainder term on the basis of an estimate for the spectral function of an 
elliptic operator for large frequencies.\\
Let $\Gamma$\ be a discrete subgroup of motions of the complex hyperbolic space
$\C H^n$, for arbitrary points $z$ and $z'$ in $\C H^n$,
we count the number of $\Gamma$-lattice points in the complex hyperbolic space
$\C H^n$ 
\begin{equation} N(T, z, z')=\#\{\gamma \in \Gamma, d(z,\gamma z')<T\},\end{equation} 
and we give the analogous of these results in the n-complex hyperbolic space.\\  
That is we investigate the local average
of the number of lattice points in complex hyperbolic ball of radius $T$,  we vary the center of the complex hyperbolic ball locally,
 and we study the function
\begin{equation}
 I(T, z, z', \alpha)=\int_F N(T, x , z') h(x) d\mu(x),
\end{equation} 
where $F=\C H^n/ \Gamma$ is the fundamental region of $\Gamma$,
and $h(x)$ is  a smooth compactly supported function satisfying 
the conditions:\\
${\bf 1)} h(x) > 0, {\bf 2)} \int h(x)d\mu(x) = 1, {\bf 3)} h(x)=0 $ for $d(x, z)\geq \alpha$ and \\  ${\bf 4)}  h(x)=O(\alpha^{-2n})$.
(For the construction of the function $h$ see  the Appendix).\\ 
 Let $u(t, z)$ be the 
solution of the Cauchy problem for the wave equation in the complex hyperbolic space $\C H^n$
 \begin{equation} \label{Cauchy-Problem}\qquad \left \{\begin{array}{cc}\partial_t^2 u(t, z)=L_nu(t, z)&(t, z)\in R\times \C H^n\\ 
u(0, z)=0 & \partial_t u(t, z)=f(z)\in C_0^\infty(\C H^n) \end{array}
\right..\end{equation}
The main results of this paper are the following theorems.
\begin{thm}\label{thm1} For sufficiently small $\alpha$,
the following formula hold\\
i)
\begin{align}I(T, z, z',\alpha)=\sum_{\gamma\in \Gamma}\int_{d(\gamma x, z')<T}h(x)d\mu (x),\end{align}
ii)  \begin{align}N(T-\alpha, z, z')\leq I(T, z, z', \alpha)\leq 
N(T+\alpha, z, z').\end{align}
iii) Let $u(t, z)$ be the solution of the Cauchy problem \eqref{Cauchy-Problem}, with initial data 
\begin{align} f(x)=\sum_{\gamma\in \Gamma}h(\gamma^{-1}x),\end{align}
 then we have \\
 $I(T, z, z', \alpha)=c_n\cosh^{1/2}T\int^T_0(\cosh T-\cosh t)^{n-3/2}\times $ \begin{equation}\label{I} F(-1/2, 3/2, n-1/2
,{\cosh T-\cosh t\over 2\cosh T})  \sinh t u (t, z') d t, \end{equation}
where 
 $c_n=(-1)^{n-1}\pi^{n-1/2}2^{n+1/2}/ \Gamma (n-1/2)$, $F(a, b, c, t)$ is the classical hypergeometric function ${}_2F_1$ given in \eqref{Gauss}.\\
\end{thm}
\begin{thm}\label{thm2} If $\Gamma$\ is a discrete subgroup of motions of the complex hyperbolic space 
$\C H^n$, $n\geq 1$, which we assume to be torsion free, 
then we have, for $\alpha$ sufficiently small,
\begin{align} I(T, z, z', \alpha)={\pi^n\over\Gamma(n+1)}\sinh^{2n}T\times\nonumber \\ \int_{-n^2}^{+\infty}F((n-
i\sqrt\lambda)/2, (n+i\sqrt\lambda)/2, n+1, -\sinh^2T) \nonumber\\
d_\lambda\int_{F} \theta_\Gamma(x, z', \lambda)h(x)d\mu(x),\end{align}
where $\theta_\Gamma(z, z', \lambda)$ is the spectral function of the Laplace-Beltrami operator 
on $F$, $F(a, b, c, t)$ is the classical hypergeometric function ${}_2F_1$ given in \eqref{Gauss}
\end{thm}
\begin{thm}\label{thm3} The hypotheses are the same as in Theorem \ref{thm1}, furthermore 
assume that
$\Gamma $ is cocompact or of finite covolume, then we have 
\begin{align}N(T, z, z')=A(T, z, z')+\left\{
\begin{array}{ll}
O(e^{(2n-1-{2n-2\over 2n+1})T})\\
 O(e^{(2n-2-{2n-4\over 2n+1})T})
\end{array}
\right.\end{align}
where
\begin{align}\label{Finite-part} A(T, z, z')=\left(\frac{\pi}{2}\right)^n\sum_{j=1}^N{2^{-\mu_j}\Gamma(\mu_j)e^{(n+\mu_j)T}\over 
\Gamma((n+\mu_j)/2)\Gamma((1+(n+\mu_j)/2)} \varphi_j(z) \varphi_j(z')\end{align} 
where $\mu_j=\sqrt{|\lambda_j|}$\ , and $ 
-n^2\leq\lambda_1\leq\lambda_2\leq\ldots\leq\lambda_N<0$, are the eigenvalues of the 
shifted Laplace-Beltrami operator $L_\Gamma$, on the interval $(-n^2, 0)$, and $ 
\varphi_j(z), 0\leq j\leq N $\ are the corresponding normalized eigenfunctions.\\
Note that the summation in \eqref{Finite-part} is meaningful as long as $$n+\mu_j > 
\left\{\begin{array}{ll}2n-1-{2n-2\over 2n+1}&{n\geq 2}\\ 2n-2-{2n-4\over 2n+1}& {n > 2} 
\end{array}
\right..$$
and consequently
$$-n^2\leq \lambda_j < -(n-1)^2({2n-1\over 2n+1})^2, \qquad n\geq 2$$ $$-n^2\leq 
\lambda_j < -(n-2)^2({2n-1\over 2n+1})^2, \qquad n > 2.$$ 
\end{thm}
In this paper we take as a model of the complex hyperbolic space the ball
$\C H^n=\{z=(z_1, z_2, ..., z_n)\in \C^n, |z| < 1 \}$ where $\C$ is the the complex field and 
$|z|^2 = z{\overline z} $, equiped with the metric $ds$ given by.
$$ ds^2=(1-
|z|^2)^{-2}\sum_{i, j=1}^n[(1-
|z|^2)\delta_{i j}+z_i\overline{z_j}]d\overline{z_i}\otimes d z_j.$$
Recall that $(\C H^n, ds)$ is a complete 
Riemannian manifold with negative sectional curvature, in which the group of motions is 
$SU(n,1)$.
 In this case, the volume 
element is given by $$ d\mu(z)=\frac{dz}{(1-|z|^2)^{n+1}},
$$
 where $dz$ is the 
Lebesgue volume element on $C^n$.\\  The $SU(n,1)-$invarian Laplace-Beltrami operator associated to the metric $ds$ is of  
the form \begin{equation}
L_n= 4(1-|z|^2)\sum_{i, j=1}^n\left(\delta_{i j}-|z|^2\right){\partial^2\over \partial z_i \partial\overline{ z_j}}+n^2.
\end{equation} 
It is known that the operator $- L_n$ is elliptic selfadjoint non-negative and has a 
continuous spectrum represented by the positive real axis.\\ 
The remanding of the paper is organized as follows, Section 2 is devoted to the integral formulas for the number of lattice points  $N(T, z, z')$ .
That is we prove the Theorems  \ref{thm1} and  \ref{thm2}.  Section 3 is devoted to  the proof of the Theorem \ref{thm3}, which gives an asymptotic formula for the number of lattice point $N(T, z, z')$ 
 with an estimate of the remainder term.
\section{Integral formulas for the number of lattice points  $N(T, z, z')$}
We prove the following lemmas
\begin{lem}\label{lem1} Let $\Gamma$ be a discrete subgroup of motions of the complex hyperbolic space, for  $h(x)$ is a smooth function satisfying 
the above conditions 1), 2) and 3), then
\begin{align}\int_{d(\gamma x, z')<T}h(x)d\mu (x)=\left\{
\begin{array}{ll}
1, \mbox{\rm if}\ \  d(z',\gamma z)\leq T-
\alpha\\
 0, \mbox{\rm if} \ \  d(z',\gamma z)\geq T+\alpha
\end{array}
\right.,\end{align}
and lies between $0$ and 
$1$ otherwise.\\
\end{lem}
\begin{proof}
 If $ d(z', \gamma z)\leq T-\alpha$\ and $d(x, z)\leq \alpha$, then
$$d(\gamma x, z')\leq d(\gamma x,\gamma z)+d(\gamma z, z')\leq \alpha + T-\alpha =T,$$ 
this shows that the support of $h$ is contained in the ball  
$\{z',  d(\gamma x, z')<T\}$.\\
Now if $ d(z', \gamma z)\geq T+\alpha $ then $$d(x, z)= d(\gamma x, \gamma z)\geq 
d(z',\gamma z)-d(\gamma x, z')\geq T+\alpha - T=\alpha .$$ \smallskip\
and therefore the above  integral is equal to zero because $d(x, z)> \alpha$ and the proof of the lemma is finished. 
\end{proof}
\begin{lem} \label{lem2}  Set  $ K(T, t)=\left({\partial\over\sinh t\partial t}\right)^{n-1} (\cosh T-\cosh t)^{n-3/2}\times$ \begin{align} \label{KT1}F(-1/2, 3/2, n-1/2, (\cosh T-\cosh t)/2\cosh T),
\end{align} 
then the following formula holds
\begin{equation}\label{KT2} K(T, t)=c_n^1\cosh^{-1/2}T\cosh t  (\cosh^2 T-\cosh^2 t)^{-1/2}\end{equation} 
with $c_n^1 =(-1)^{n-1}\Gamma(n-1/2)\frac{\sqrt{2}}{\sqrt{\pi}}$.
\end{lem} 
\begin{proof}
Set $z= (\cosh T-\cosh t)/2\cosh T$, ${\partial\over\sinh 
t\partial t}=\frac{-1}{2\cosh T}\frac{\partial }{\partial z}$,\\
$K(T, t)=(-1)^{n-1}(2\cosh T)^{-1/2}\times$\begin{align}\left(\frac{\partial}{\partial z}\right)^{n-1}[z^{n-3/2}F(-1/2, 3/2, n-1/2, z)]. \end{align}
Using the
formula in (Magnus et al. \cite{Magnus et al.} p. 41) 
\begin{equation} \label{derivative}{d^m\over dy^m}y^{c-1}F(a, b, c; y)=(c-m)_m y^{c-m-1}F(a, b, c-m, y)\end{equation} we can write 
$K(T, t)=(1/2)_{n-1}(-1)^{n-1}(2\cosh T)^{-1/2}\times$\begin{align} z^{-1/2}F(-1/2, 3/2, 1/2, z). \end{align}
Using the following relation (Magnus et al. \cite{Magnus et al.} p. 50) $$F(a,1-a, c, y)=(1-y)^{c-1} F((c-
a)/2, (c+a-1)/2, c, 4y-4y^2),$$
$K(T, t)=(1/2)_{n-1}(-1)^{n-1}(2\cosh T)^{-1/2}\times$\begin{align} z^{-1/2}(1-z)^{-1/2}F(1/2, -1/2, 1/2, 4z-4z^2). \end{align}
and the formula \cite{Magnus et al.} p. 38) $F(a, b, b,  z)=(1-z)^{-a}$,\\
$K(T, t)=(1/2)_{n-1}(-1)^{n-1}\sqrt{2}(\cosh T)^{-1/2} \cosh t  (\cosh^2 T-\cosh^2 t)^{-1/2}$\\
$K(T, t)=c_n^1(\cosh T)^{-1/2} \cosh t  (\cosh^2 T-\cosh^2 t)^{-1/2}$ as stated,
and the proof of Lemma \ref{lem2} is finished.
\end{proof}
\begin{lem}\label{lem3}
Set $I(T, z')=c_n\cosh^{1/2}T\int^T_0(\cosh T-\cosh t)^{n-3/2}\times $ \begin{equation}\label{I} F(-1/2, 3/2, n-1/2, {\cosh T-\cosh t\over 2\cosh T})  \sinh t u (t, z') d t, \end{equation}
where $u(t, z')$ is the solution of the Cauchy problem for the wave equation on the complex hyperbolic space,  
 $c_n=(-1)^{n-1}\pi^{n-1/2}2^{n+1/2}/ \Gamma (n-1/2)$, $F(a, b, c, t)$ is the classical hypergeometric function ${}_2F_1$ given in \eqref{Gauss},\\
 then we have 
\begin{align}I(T, z')=\int_{d(z', x)<T}f(x)d\mu(x).\end{align} 
\end{lem}
\begin{proof}
Recall that the Cauchy problem \eqref{Cauchy-Problem} has a unique solution given by (Intissar-Ould Moustapha \cite{Intissar-Moustapha}): 
\begin{equation}\label{Solution-Wave} u(t, z')=(2\pi)^{-n}\left({\partial\over\sinh t\partial t}\right)^{n-1}\int_{d(z', x)<t}\, {f(x)\, d\mu(x)\over 
\sqrt{\cosh^2t-\cosh^2d(z', x)}},\end{equation} where $d(z', x)$ is the geodesic distance between $z'$ and 
$x$ in $\C H^n$ and $d\mu(x)$ is the volume element on $\C H^n$.\\
Inserting the expression of $u(t, z')$ given by \eqref{Solution-Wave}in the integral \eqref{I} and 
integrating by parts $(n-1)$-times, we obtain \\ 
 \begin{align} I(T, z')=c_n(2\pi)^{-n}(\cosh T)^{1/2} \int_0^T 
K(T, t)\times \nonumber\\ \int_{d(z', x)<t}\frac{f(x)}{\sqrt{\cosh t-\cosh^2d(z', x)}}d\mu(x)\sinh t dt,\end{align}
 where $K(T, t)$ is as in \eqref{KT1}
Using the formula \eqref{KT2} and changing the order of integration we obtain\\
$ I(T,  z')= c_n c_n ^1 (2\pi)^{-n} \int_0^T (\cosh^2 T-\cosh^2 t)^{-1/2}\times$
 \begin{align} \int_{d(z', x)<t}\frac{f(x)}{\sqrt{\cosh^2t-\cosh^2d(z', x)}}d\mu(x)\sinh t\cosh t dt\end{align}
$$ I(T, z')= \frac{2}{\pi}\int_{d(z', x)<t} f(x) j(T, r) d\mu(x)$$
where\\
$$j(T, r)=\int_{r}^T(\cosh^2 T-\cosh^2 t)^{-1/2}(\cosh^2t-\cosh^2r)^{-1/2}\sinh t\cosh t dt,$$
putting, $s=\cosh^2t-\cosh^2r$ we obtain
$$j(T, r)=\int_{r}^T(\cosh^2 T-\cosh^2 t)^{-1/2}(\cosh^2t-\cosh^2r)^{-1/2}\sinh t\cosh t dt$$
and $$j(T, r)=\int_0^{\cosh^2 T-\cosh^2 t}(\cosh^2 T-\cosh^2 r-s^2)^{-1/2}s^{-1/2}\frac{ds}{2}$$
Making use of the substitution  $s=(\cosh^2 T-\cosh^2 t)z$, we have\\
$j(T, r)=\int_0^1(1-z)^{-1/2} z^{-1/2}\frac{dz}{2}=\frac{1}{2}\beta(1/2, 1/2)=\frac{\pi}{2}$,\\ and the proof of the Lemma is finished.
\end{proof}
\begin{lem} \label{lem4} 
 Set
$$H_n(\lambda, T)= \sinh ^{2n}T F((n-i\sqrt\lambda)/2, (n+i\sqrt\lambda)/2, n+1, -\sinh^2T)$$
then we have\\
 $ H_n(\lambda, T)=\int_0^T(\cosh T-\cosh t)^{n-1/2}\times$
\begin{align}\label{Int-Rep} F(-{1\over 2},{3\over 2}, n+{1\over 2}, {\cosh T-\cosh t)\over 2\cosh T}) 
\cos\sqrt{\lambda} t dt.\end{align}
\end{lem}
\begin{proof}
 This lemma is direct consequence of the generalized Mehler Fucs formula for the Jacobi function (Koornwinder\cite{Koornwinder}) \begin{equation}\label{Mehler-Fucs} 
[\Gamma (\alpha +1)]^{-1}\Delta (T)\varphi^{(\alpha,\beta)}_{\sqrt\lambda} (T) =\pi^{-1/2}\int_0^T\cos\sqrt\lambda t E(t,T) dt, \end{equation} where
 $\Delta(T)=\sinh^{2\alpha + 1}T\cosh^{2\beta + 1}$ and
\begin{equation} E(t, T)=c_{\alpha}\sinh 2T\cosh^{\beta -1/2}T(\cosh T-\cosh t)^{\alpha -1/2}\times $$ $$ 
F(1/2+\beta, 1/2-\beta, \alpha+1/2,  (\cosh T-\cosh t)/2\cosh T),\end{equation} and $c_{\alpha}= 
2^{\alpha-1/2}[\Gamma (\alpha+1/2)]^{-1}$ and
$$\varphi^{(\alpha, \beta)}_\lambda (x)={}_2F_1\left(\frac{\alpha+\beta+1-i\lambda}{2}, \frac{\alpha+\beta+1+i\lambda}{2}, \alpha+1, -\sinh^2 x\right).$$
\end{proof}
Proof of the Theorem \ref{thm1}
To prove  i)
$$N(T, z, z')=\sum_{\gamma\in\Gamma}\chi_{B(z, T)}(\gamma z')
$$
where $\chi_{B(z, T)}$ is the characteristic function of the ball $B(z, T)$ is the complex hyperbolic ball of center $z$ and radius $T$
$$
I(T, z, z', \alpha)=\int_F N(T, x, z')h(x)d\mu(x)=\int_F \sum_{\gamma\in\Gamma}\chi_{B(x, T)}(\gamma z')h(x)d\mu(x)
$$
$$
I(T, z, z', \alpha)=\sum_{\gamma\in\Gamma}\int_{d(x, \gamma z') < T} h(x)d\mu(x)
$$
$$
I(T, z, z', \alpha)=\sum_{\gamma\in\Gamma}\int_{d( \gamma x, z') < T} h(x)d\mu(x)
$$
and hence 
$$
I(T, z, z', \alpha)=\sum_{\gamma\in\Gamma}\int_{d( \gamma x, z') < T} h(x)d\mu(x)
$$
and this prove i).\\
The part ii) is a consequence of the i) and the Lemma \ref{lem1}.\\
Finally iii) is a consequence of the Lemma \ref{lem3} and the part i). The proof of Theorem \ref{thm1} is finished.\\
Proof of the Theorem \ref{thm2}\\
We note that if the initial data 
$f$ of the problem \eqref{Cauchy-Problem} is $\Gamma-$\ automorphic, then the solution $u(t, z')$ has the 
same property. This follows from the uniqueness of the solution and the invariance of 
Laplace-Beltrami operator under motions group $SU(n,1)$.\\ On the other hand the 
solution of the
Cauchy problem \eqref{Cauchy-Problem} for the initial data $f(x)=\sum_{\gamma\in \Gamma} h(\gamma^{-1}x)$, can be represented
for a sufficiently small $\alpha $\ in the form  $$u(t, z')=\int_{-n^2}^{+\infty} {\sin 
\sqrt{\lambda t}\over \sqrt{\lambda}}d_\lambda 
\int_{F}h(x)\theta_\Gamma(x, z', \lambda) d\mu (x).$$
 Substituting this in 
\eqref{I}, we obtain after integrating by parts and changing the order of integration:\\
$I(T, z, z', \alpha)=c_n\cosh ^{1/2}T\int_{-n^2}^{+\infty}\int_0^T(\cosh T-\cosh t)^{n-1/2}\times$ \begin{align}\label{III} F(-{1\over 
2},{3\over 2}, n+{1\over 2}, {\cosh T-\cosh t)\over 2\cosh T}) 
\cos\sqrt{\lambda} t\,  dt\, \nonumber\\  d_\lambda \int_F h(x)\theta_\Gamma(x, z', \lambda) d\mu(x),\end{align}
 and the proof of the theorem \ref{thm2} is finished.
Note that  iii) is analogous of the formulas in  the Lemma 2.5  of ( Lax-Phillips \cite{Lax-Phillips} p.318). 
\section{The asymptotic behavior of the number $N(T, z, z')$}
It is clear, 
from the inequalities ii) of the theorem \ref{thm2}, which  equivalent to
\begin{align}I(T-\alpha, z, z')\leq N(T, z, z', \alpha)\leq 
I(T+\alpha, z, z', \alpha),\end{align}
that is in order to study the asymptotic behavior of the number 
$N(T, z, z')$ it suffices to study that of $I(T, z, z', \alpha)$.\\
For this set
\begin{align}\label{IT} I(T, z, z', \alpha)=I_1 + I_2\end{align}
 where\\
$  I_1={\pi^n\over \Gamma (n+1)}\sinh ^{2n}T\times$\\ 
\begin{align} \sum_{j=1}^N F({n+\mu_j\over 
2}, {n+\mu_j\over 2}, n+1, -\sinh^2T)\varphi_j(z')\int_{F} 
\varphi_j(x)h(x)d\mu(x) \end{align}  with $\mu_j=\sqrt{|\lambda|}$\ and\\
$ I_2={\pi^n\over\Gamma (n+1)}\sinh^{2n}T\int_0^{+\infty}F((n-
i\sqrt\lambda)/2, (n+i\sqrt\lambda)/2, n+1, -\sinh^2T)\times$
\begin{align} d_\lambda\int_{F}\theta_\Gamma(x, z', \lambda)h(x)d\mu(x).\end{align} To 
estimate $I_1$ we use the formula relating hypergeometric functions of arguments $z$ and $1/z$ (Magnus et al.\cite{Magnus et al.})p.48
\begin{align} F(a, b, c, z)=\frac{\Gamma(c)\Gamma(b-a)}{\Gamma(b)\Gamma(c-a)}(-z)^{-a}F(a, a-c+1, a-b+1,1/z)+\nonumber\\ \frac{\Gamma(c)\Gamma(a-b)}{\Gamma(a)\Gamma(c-b)}
(-z)^{-b}F(b, b-c+1, b-a+1,1/z)\end{align}
and we obtain for $T\rightarrow +\infty$\ : \begin{align} F({n+\mu_j\over 2}, {n-\mu_j\over 
2}, n+1, -\sinh^2T)=c_j(n) e^{(\mu_j-n)T}[1+O(e^{-2T})]\end{align} \begin{align}c_j(n)={\Gamma 
(n+1)\Gamma (\mu_j)2^{-\mu_j+n}\over \Gamma((n+\mu_j)/2)\Gamma(1+(n+\mu_j)/2)}\end{align} 
replacing in $(2.2)$ we get:
\begin{align}\label{I1}I_1= ({\pi\over 2})^n\sum_{j=1}^N{\Gamma(\mu_j)2^{-\mu_j}\over \Gamma 
((n+\mu_j)/2)\Gamma(1+(n+\mu_j)/2)} e^{(n+\mu_j)T}\varphi_j(z')\nonumber\\ \int_{F} 
\varphi_j(x)h(x)d\mu(x)+O(\alpha e^{(2n-2)T})\end{align}
In order to 
estimate $I_2$ we need the following lemmas\\  
\begin{lem} \label{lem5} 
 Let $H_n(\lambda, T)$ be as  in Lemma \ref{lem4}
then we have\\
for $\lambda \rightarrow \infty$ and $T\rightarrow \infty $, the following estimates hold \\
\begin{align}H_n(\lambda, T)=O(\lambda ^{-(2n+1)/4}e^{n T}) 
\end{align}
\end{lem}
\begin{proof}
We remark that  for $0\leq\frac{\cosh T-\cosh t}{2\cosh T}\leq \frac{1}{2}$  and hence the hypergeometric in the R.H.S. of  \eqref{Int-Rep} is bounded,
 \begin{align}H_n(\lambda, T)=O(\int_0^{T}(\cosh T - \cosh t)^{(2n-1)/2}\cos t\sqrt \lambda d t) \end{align}
using  the formula (2.14) in (Levitan \cite{Levitan}, p. 27)
$$ \int_0^{T}(\cosh T - \cosh t)^{(m-1)/2}\cos t\sqrt \lambda d t =O(\lambda^{-(m+1)/4}e^{(m-1)/4 T})$$
we have the result of the Lemma.
\end{proof}
Proof of Theorem \ref{thm3}
From \eqref{III} and  Lemma \ref{lem4}, we 
have:$$I_2=O(e^{nT}\int_{F}h(x)\int_0^1|d_\lambda\theta_\Gamma(x, z',
\lambda)|d\mu(x))+$$ $$\hspace{\fill}+O(e^{nT}\int_1^{+\infty}\lambda ^{-{2n+1\over 
4}}|d_\lambda\int_{F} \theta_\Gamma(x, z', \lambda)h(x)d\mu(x)|)$$ 
$$I_2=O(e^{nT})+O(e^{nT}\int_1^{+\infty}\lambda ^{-{2n+1\over 
4}}|d_\lambda\int_{F} \theta_\Gamma(x,z',\lambda)h(x)d\mu(x)|)$$ set :
$$J=\int_1^{+\infty}\lambda ^{-{2n+1\over 4}}|d_\lambda\int_{F} 
\theta_\Gamma(x, z', \lambda)h(x)d\mu(x)| $$ let $\delta > 0$, set :
\begin{equation} J=J_1+J_2\end{equation}
where
$$J_1=\int_1^{e^{\delta T}}\lambda ^{-{2n+1\over 4}}|d_\lambda\int_{F} 
\theta_\Gamma(x,z',\lambda)h(x)d\mu(x)|$$
and
$$ J_2=\int_{e^{\delta T}}^{+\infty}\lambda ^{-{2n+1\over 4}}|d_\lambda\int_{F} 
\theta_\Gamma(x, z', \lambda)h(x)d\mu(x)|$$
Hence we have
\begin{align}\label{I2} I_2=O(e^{nT})+O(e^{nT}J_1)+O(e^{nT}J_2).\end{align}
 Next we use the inequality 
as in \cite{Levitan} page $29$: \begin{equation} |\Delta\theta(x, z', \lambda)|\leq {1\over 
2}[\Delta\theta(x, x, \lambda)+\Delta\theta(z', z', \lambda) ] \end{equation}  to obtain:
$$J_1\leq {1\over 2}\int_1^{e^{\delta T}}\lambda ^{-{2n+1\over 4}} 
d\theta(x, x, \lambda)+{1\over 2}\int_1^{e^{\delta T}}\lambda ^{-{2n+1\over 4}} 
d\theta(z', z', \lambda) $$ and by the estimate \cite{Hormander}:
\begin{equation}\label{Estimate} \theta(x, x, \lambda)= O(\lambda^n) \end{equation} we have:
\begin{equation} \label{J1}J_1=O(e^{({2n-1\over 4})\delta T}).\end{equation} For the estimation of $J_2$, we recall 
the expansion of the spectral function of an elliptic operator in terms of its eigenfunctions: 
\begin{equation} \label{Theta} \theta_\Gamma (x, z', \lambda)=\int_{-n^2}^\lambda \sum_{j=1}^{N(\nu)}\varphi_j(\nu , 
z') \overline{\varphi_j(\nu, x)}d\rho (\nu) \end{equation} where $1\leq N(\nu)\leq \infty 
$,$\varphi_j(x,\nu)$\ is a $\Gamma-$\ automorphic solution of the equation $L\varphi = 
\nu\varphi_j$\ and $d\rho(\nu)$\ is a non decreasing function (see Berezanskii \cite{Ju} and Gel'fand- Kostyuchenko \cite{Gel'fand- Kostyuchenko}).\\
Recall that in this case, Parseval's equation is given by: 
\begin{equation}\label{Parsevall} \int_{F}|f(x)|^2d\mu(x)=\int_{-
n^2}^{+\infty}\sum_{j=1}^{N(\lambda)}|\int_{F}f(z')\overline{\varphi_j(z',\nu)}d\mu(z')|^2d \rho(\lambda) \end{equation}
From \eqref{Theta} we have:
$$J_2 \leq \int_{e^{\delta T}}^{+\infty}\lambda ^{-{2n+1\over 
4}}\sum_{j=1}^{N(\lambda)}|\varphi_j(z',\lambda)||\int_{F} 
h(x)\overline{\varphi_j(x,\lambda)}d\mu(x)|d\rho(\lambda).$$ Using the Cauchy-Shwartz 
inequality twice, we obtain: $$J_2\leq \int_{e^{\delta T}}^{+\infty}\lambda^{-{2n+1\over 
4}}(\sum_{j=1}^{N(\lambda)}|\varphi_j(z',\lambda)|^2)^{1/2}(\sum_{j=1}^{N(\lambda)}|\int_
{F} h(x)\overline{\varphi_j(x,\lambda)}d\mu(x)|^2)^{1/2}d\rho(\lambda)$$ 
$$\hspace{\fill}\leq(\int_{e^{\delta T}}^{+\infty}\lambda ^{-{2n+1\over 
2}}\sum_{j=1}^{N(\lambda)}|\varphi_j(z',\lambda)|^2 d\rho(\lambda))^{1/2}(\int_{e^{\delta 
T}}^{+\infty}\sum_{j=1}^{N(\lambda)}|\int_{F} 
h(x)\overline{\varphi_j(x,\lambda)}d\mu(x)|^2d\rho(\lambda))^{1/2}$$ and by using  \eqref{Theta} and  \eqref{Parsevall}
we obtain: $$J_2\leq (\int_{e^{\delta 
T}}^{+\infty}\lambda^{-{2n+1\over 
2}}d_\lambda\theta(z', z', \lambda))^{1/2}(\int_{F}h^2(x)d\mu(x))^{1/2}.$$ 

Using the estimates \eqref{Estimate} and the formula $\int  h^2(x)d\mu(x) = O(\alpha^{-2 n})$, we have:
 \begin{equation}\label{J2}J_2=O(e^{-\delta T/4}\alpha^{-n})\end{equation}
 It follows from \eqref{IT}, \eqref{I1},\eqref{I2}, \eqref{J1}  and \eqref{J2} that: $$I(T \pm 
\alpha, z_0, z, \alpha)=A(T, z_0, z)+O(\alpha e^{2(n-1)T})+O(e^{[n+({2n-1\over 2})\delta] 
T})+O(e^{({n-\delta\over 4})T}\alpha^{-n})$$ Putting $\alpha=e^{-\epsilon} ,\epsilon > 0$\ 
the best ratio between $\epsilon $\ and $\delta $\ is obtained by satisfying the equation:
$$n + {2n-1\over 4}\delta = (n-\delta/4) + n\epsilon =2n - 1 - \epsilon $$ from which it 
follows that for $n\geq 2$\ 
 $$\delta=2\epsilon ,\qquad \epsilon={2(n-1)\over 2n+1} 
,\qquad 2n-1-\epsilon = 2n-1- {2(n-1)\over 2n+1} $$ and for $n > 2$
$$\delta=2\epsilon, \qquad \epsilon={2(n-2)\over 2n+1}, \qquad 2n-1-\epsilon = 2n-2- 
{2(n-2)\over 2n+1} $$ This completes the proof of Theorem \ref{thm3}.
\section{Appendix}
Let $z$ be a fixed point in the complex hyperbolic space $\C H^n$, and let $h_1$ be a smooth
function whose support is concentrated in the complex hyperbolc  ball of radius 1
and centre  $z$, and suppose that
\begin{align} \int h_1(x)d\mu(x) = 1
\end{align} 
 Using the geodesic polar coordinates on the the complex hyperbolic space $\C H^n$,  $x=\tanh r \omega $, $r\geq 0$  and $\omega\in S^{2n-1}$, for $\alpha$  a small positive number, put
\begin{align} h(x) = c_n(\alpha) h_1(\frac{r}{\alpha}, \theta), r=d(x, z),\end{align}
where the constant $c_n(\alpha)$ is chosen so that
\begin{align} \int h(x)d\mu(x) = 1.
\end{align} 
 We show that the constant $c_n(\alpha)$ satisfies the condition
$$\lim _{\alpha\rightarrow 0}\alpha^{2n} c_n(\alpha) = 1.$$
\begin{align} 1=c_m(\alpha)\int_{S^{2n-1}}d\theta\int_0^\alpha h_1(\frac{r}{\alpha}, \theta)\sinh^{2n-1}r\cosh r d r
\end{align}
\begin{align} 1=c_m(\alpha)\alpha \int_{S^{2n-1}}d\theta\int_0^1 h_1(t, \theta)\sinh^{2n-1}\alpha t \cosh \alpha t d t
\end{align}
\begin{align} 1=\alpha^{2n}c_m(\alpha)\int_{S^{2n-1}}d\theta\int_0^1 h_1(t, \theta)\left(\frac{\sinh\alpha t}{\alpha \sinh  t}\right)^{2n-1}\sinh^{2n-1} t\cosh t d t
\end{align}
\begin{align} 1=\alpha^{2n}c_m(\alpha)[\int_{S^{2n-1}}d\omega\int_0^1 h_1(t, \omega)\sinh^{2n-1} t\cosh t d t+o(1)]
\end{align}
\begin{align}1=\alpha^{2n}c_n(\alpha)[1+o(1)].\end{align}

 Department of Mathematics, College of Arts and Sciences\\ Gurayat, Jouf University-Kingdom of Saudi Arabia.\\
 Facult\'e des Sciences et Techniques, Universit\'e de \\  Nouakchott Al-Aasriya, Nouakchott-Mauritanie \\ mohamedvall.ouldmoustapha230@gmail.com
\end{document}